\documentclass[12pt]{amsart}

\usepackage{amsthm}
\usepackage{amsmath}
\usepackage{amssymb}
\usepackage{mathrsfs}
\usepackage[dvipsnames]{xcolor}
\usepackage{tikz-cd}
\usepackage{enumitem}
\usepackage{hyperref}
\usepackage{mathtools}
\hypersetup{colorlinks = true,
	linkcolor = MidnightBlue,
	urlcolor = BrickRed,
	citecolor = MidnightBlue}

\title[Generic geometrically elliptic normal curves]{The period and index of a generic geometrically elliptic normal curve}
\author{Eoin Mackall}
\email{eoinmackall \emph{at} gmail.com}
\urladdr{\url{www.eoinmackall.com}}
\date{October 27, 2025}
\keywords{Hilbert schemes; division algebras}
\subjclass{14H45; 14C05; 16K20}

\newtheorem{thm}{Theorem}[section]

\newtheorem{prop}[thm]{Proposition}
\newtheorem{cor}[thm]{Corollary}
\newtheorem{lem}[thm]{Lemma}

\theoremstyle{definition}
\newtheorem{defn}[thm]{Definition}
\newtheorem{exmp}[thm]{Example}
\newtheorem{rmk}[thm]{Remark}

\newcounter{item}

\newcommand{\HH}{\mathrm{H}}

\newcommand{\CH}{\mathrm{CH}}

\newcommand{\SB}{\mathbf{SB}}

\begin{document}
	\begin{abstract}
	We construct genus one curves on base extensions of generic Severi--Brauer varieties of a given index and period which are versal objects for families of geometrically elliptic normal curves. We also compute the periods and indices of these curves showing that all possible period/index combinations are possible.
	\end{abstract}
	\maketitle
	\tableofcontents
	
	\section{Introduction}
	In \cite{MR0106226}, Lang and Tate introduced the notion of the period and index of a principal homogeneous space for an abelian variety in analogy to invariants of the same name for central simple algebras. They prove that, here also, the period divides the index, that both invariants have the same prime factors, and they construct examples showing some of the possible period-index combinations that can occur.
	
	Lichtenbaum showed in \cite{MR0242831} that for principal homogeneous spaces under elliptic curves, the index divides the period squared. Since then, there have been a number of constructions of genus one curves having period $n$ and index $nm$, for any $m$ dividing $n$, for varying base fields. For instance, Clark and Lacy \cite{MR3935899} have proven that such curves exist over any infinite, finitely generated field and for any $n\geq 1$.
	
	In this article, we show that there exists a generic, geometric such construction of a curve of genus one with period $n\geq 3$ and index $nm$, for any $m$ dividing $n$, assuming that $n$ is indivisible by the characteristic of the base field. More precisely, we compute (in Theorem \ref{thm: perind}) the period and index for the generic geometrically elliptic normal curve on a generic Severi--Brauer variety of index $n$ and exponent $m$.
	
	Our proof has two components. First, we observe that the generic geometrically elliptic normal curve embedded inside the generic Severi--Brauer variety of index $n$ and exponent $m$ is \textit{versal} among all curves embedded as a geometrically elliptic normal curve in a Severi--Brauer variety of degree $n$ and exponent dividing $m$. More precisely, we can realize any particular such embedded curve as a specialization of the generic curve along a sequence of DVR's.
	
	Second, the period and index of any geometrically elliptic normal curve embedded in a Severi--Brauer variety of degree $n$ and exponent $m$ must be bounded above by $n$ and $nm$ respectively. Since the period and index can only lower under specialization, the main difficulty is in showing that these upper bounds are also lower bounds. To do this, we use the idea of \textit{index reduction for curves} to produce a particular example which the generic curve specializes to (Lemma \ref{lem: indcurve}).
	
	Our construction of the generic, geometric elliptic normal curve inside a generic Severi--Brauer variety, along with the verification of its properties, uses a particular \textit{twisted Hilbert scheme} of a \textit{Severi--Brauer scheme}. Conceptually, it is easier to understand this twisted Hilbert scheme as an fppf-descended collection of Hilbert schemes of projective bundles. We explain how this descent can be accomplished in Section \ref{sec: descent} and we summarize the basic properties of these schemes that we will use. This section subsumes an earlier preprint of the author titled \textit{Twisted Hilbert schemes and division algebras}. Then, in Section \ref{sec: ggenc}, we give our construction of generic geometrically elliptic normal curves along with the main results of the paper.\\
		
	\noindent\textbf{Notation}. We use the following notation throughout:
	\begin{itemize}[leftmargin=*]
		\item if $k$ is a base field, then we write $\overline{k}$ to denote a fixed algebraic closure of $k$ and $k^s$ to denote the separable closure of $k$ inside $\overline{k}$.
	\end{itemize}
	
	\noindent\textbf{Conventions}. We use the following conventions throughout:
	\begin{itemize}[leftmargin=*]
		\item a variety is an integral scheme that is separated and of finite type over a base field,
		\item a curve is a proper scheme of pure dimension one that is separated and of finite type over a base field.\\
	\end{itemize}
	
	\noindent\textbf{Acknowledgments}. I'd like to thank both Nitin Chidambaram and Priyankur Chaudhuri for our frequent meetings discussing the Hilbert scheme where I learned most of the techniques contained in this paper. I'd also like to thank Patrick Brosnan for stimulating conversations that gave me both the ideas and motivation needed to start this work.
	
	\section{Twisted Hilbert Schemes}\label{sec: descent}
	Let $\mathscr{X}/S$ be a Severi--Brauer scheme of relative dimension $n$ over a Noetherian scheme $S$. Concretely, this means there exists an fppf cover $S'=\{S_i\}_{i\in I}$ of $S$ and compatible isomorphisms $\mathscr{X}_{S_i}=\mathscr{X}\times_S S_i\cong \mathbb{P}^n_{S_i}$. We call data $(S_i,\epsilon_i)_{i\in I}$ consisting of an fppf cover $S'$ and isomorphisms $\epsilon_i:\mathscr{X}_{S_i}\rightarrow \mathbb{P}^n_{S_i}$ a \textit{splitting of $\mathscr{X}/S$}.
	
	Given both splitting data $(S_i,\epsilon_i)_{i\in I}$ for a Severi--Brauer scheme $\mathscr{X}/S$ and a polynomial $\phi(t)\in \mathbb{Q}[t]$, one gets Hilbert schemes $\mathbf{Hilb}_{\phi(t)}(\mathbb{P}^n_{S_i}/S_i)$, see for example \cite[Chapter I, I.1]{MR1440180}, and one also gets an induced fppf descent datum relative to the cover $\{S_i\}_{i\in I}$ of $S$. The goal of this section is to show that this descent data is effective, coming naturally from an $S$-scheme $\mathbf{Hilb}^{\text{tw}}_{\phi(t)}(\mathscr{X}/S)$ which represents a functor analogous to the usual Hilbert scheme of a projective bundle.

	To start, recall from \cite[\S8.4]{MR0338129} that Quillen has constructed a universal vector bundle $\mathcal{J}$ on the Severi--Brauer scheme $\mathscr{X}/S$ having the following property: locally for an fppf cover $S'/S$ splitting $\mathscr{X}/S$, $\mathcal{J}$ admits isomorphisms \begin{equation*}\mathcal{J}|_{S_i}\cong \mathcal{O}_{\mathbb{P}^n_{S_i}}(-1)^{\oplus n+1}\quad \mbox{for each } S_i\in S'
	\end{equation*}
	compatible with the isomorphisms $\mathscr{X}_{S_i}\cong \mathbb{P}^n_{S_i}$ of the splitting. We write $\mathcal{Q}=\mathcal{J}^\vee=\mathcal{H}om(\mathcal{J},\mathcal{O}_{\mathscr{X}})$ to denote the dual of $\mathcal{J}$ and we call $\mathcal{Q}$ the \textit{Quillen bundle} on the Severi--Brauer scheme $\mathscr{X}/S$.
	
	\begin{lem}\label{lem: redh}
		Suppose that $S$ is connected and write $\pi:\mathscr{X}\rightarrow S$ for the structure map of $\mathscr{X}/S$. Let $\mathcal{F}$ be an $S$-flat coherent sheaf on $\mathscr{X}$. Then there exists a numerical polynomial $\phi(t)\in \mathbb{Q}[t]$ and an integer $N$ so that the following equality holds \begin{equation*} \mathrm{rk}(\pi_*(\mathcal{F}\otimes \mathcal{Q}^{\otimes r}))=\phi(r)\cdot \mathrm{rk}(Q^{\otimes r})
		\end{equation*}
		for all integers $r\geq N$.
	\end{lem}
	
	\begin{proof}
		Let $S'=\{S_i\}_{i\in I}$ be an fppf cover splitting $\mathscr{X}/S$ and write $\pi_i:\mathscr{X}_{S_i}\rightarrow S_i$ for map coming from base change. Then, for all $r\geq 1$, there are isomorphisms \[\pi_*(\mathcal{F}\otimes \mathcal{Q}^{\otimes r})|_{S_i} \cong \pi_{i*}(\mathcal{F}|_{S_i}\otimes (\mathcal{O}_{\mathbb{P}^n_{S_i}}(1)^{\oplus n+1})^{\otimes r})\cong \pi_{i*}(\mathcal{F}|_{S_i}(r)^{\oplus (n+1)^r}).\] Since $\pi_{i*}(\mathcal{F}|_{S_i}(r)^{\oplus (n+1)^r})\cong\pi_{i*}(\mathcal{F}|_{S_i}(r))^{\oplus (n+1)^r}$, the $\phi(t)$ of the lemma is necessarily the Hilbert polynomial of $\mathcal{F}|_{S_i}$ on $\mathscr{X}_{S_i}\cong \mathbb{P}^n_{S_i}$.
	\end{proof}
	
	\begin{defn}
		Let $\mathscr{X}/S$ be a Severi--Brauer scheme over a base $S$. Let $\mathcal{F}$ be an $S$-flat coherent sheaf on $\mathscr{X}$. For each connected component $S_\rho\subset S$ we define the \textit{reduced Hilbert polynomial of $\mathcal{F}$ on $S_\rho$} to be the numerical polynomial $\mathrm{rh}^\rho_{\mathcal{F}}(t)\in \mathbb{Q}[t]$ guaranteed to exist by Lemma \ref{lem: redh}. In other words, $\mathrm{rh}^\rho_\mathcal{F}(t)$ is uniquely characterized by the existence of an integer $N\geq 0$ and equality \[\mathrm{rk}(\pi_*(\mathcal{F}\otimes \mathcal{Q}^{\otimes t})|_{S_{\rho}})=\mathrm{rh}^\rho_{\mathcal{F}}(t)\cdot \mathrm{rk}(Q^{\otimes t})\quad \mbox{for all $t\geq N$.}\] If the reduced Hilbert polynomial of $\mathcal{F}$ is identical over all connected components $S_\rho\subset S$, then we simply write $\mathrm{rh}_{\mathcal{F}}(t)$ for this polynomial and call it the \textit{reduced Hilbert polynomial of $\mathcal{F}$}.
		
Also, in the special case when $\mathcal{F}=\mathcal{O}_V$ is the structure sheaf of a subscheme $V\subset \mathscr{X}$, we write $\mathrm{rh}_V(t)$ instead of $\mathrm{rh}_{\mathcal{O}_V}(t)$.
	\end{defn}
	
	\begin{rmk}\label{rem: rhilbsplit}
		If $\mathscr{X}/S$ is a split Severi--Brauer scheme (i.e.\ if $\mathscr{X}/S$ is isomorphic over $S$ with a projective bundle $\mathbb{P}_S(\mathcal{E})$ for some vector bundle $\mathcal{E}$ on $S$) then, for any $S$-flat coherent sheaf $\mathcal{F}$ on $\mathscr{X}$, the reduced Hilbert polynomial $\mathrm{rh}_{\mathcal{F}}(t)$ is just the usual Hilbert polynomial $h_{\mathcal{F}}(t)$ with respect to the line bundle $\mathcal{O}_{\mathbb{P}_S(\mathcal{E})}(1)$.
	\end{rmk}
	
	\begin{lem}\label{lem: exist}
		Let $\mathscr{X}/S$ be a Severi--Brauer scheme over any scheme $S$. Let $\mathcal{F}$ be a coherent sheaf on $\mathscr{X}$. Then for every polynomial $\phi(t)\in \mathbb{Q}[t]$ there is a locally closed subscheme $S_{\phi(t)}\subset S$ with the property:
		\begin{itemize}[leftmargin=*, label=\normalfont{(\textbf{f})}]\label{propf}
			\item given a morphism $T\rightarrow S$, the pullback $\mathcal{F}_T$ on $\mathscr{X}_T$ is flat over $T$ with reduced Hilbert polynomial $\mathrm{rh}_{\mathcal{F}_T}(t)=\phi(t)$ if and only if $T\rightarrow S$ factors $T\rightarrow S_{\phi(t)}\subset S$.
		\end{itemize}
	\end{lem}
	
	\begin{proof}
		The lemma holds fppf locally over the base $S$. More precisely, let $S'=\{S_i\}_{i\in I}$ be any fppf cover splitting $\mathscr{X}/S$, with $I$ a finite set, and let $\epsilon_i:\mathscr{X}_{S_i}\rightarrow \mathbb{P}^n_{S_i}$ be isomorphisms realizing the splitting. Write $T_i=T\times_S S_i$ and $\mathcal{F}_{i}$ for the pullback of $\mathcal{F}$ to $\mathscr{X}_{T_i}$. Then for each of the indices $i\in I$, there exists a locally closed subscheme $S_{i,\phi(t)}\subset S_i$ with the property that $\mathcal{F}_i$ is flat over $T_i$ with reduced Hilbert polynomial $\mathrm{rh}_{\mathcal{F}_i}(t)=\phi(t)$ if and only if $T_i\rightarrow S_i$ factors $T_i\rightarrow S_{i,\phi(t)}\subset S_i$.  Indeed, because of Remark \ref{rem: rhilbsplit}, the reduced Hilbert polynomial $\mathrm{rh}_{\mathcal{F}_i}(t)$ is just the Hilbert polynomial $h_{\epsilon_{i*}\mathcal{F}_i}(t)$; this existence is then \cite[Theorem I.1.6]{MR1440180} whose proof is ultimately deferred to \cite[Lecture 8]{MR0209285}.
		
		To see that the lemma also holds over $S$, we note that it's possible to descend the $S_{i,\phi(t)}$ to a scheme $S_{\phi(t)}\subset S$ with $S_{\phi(t)}\times_S S_i=S_{i,\phi(t)}$. Indeed, both of the schemes $S_{i,\phi(t)}\times_S S_j$ and $S_{j,\phi(t)}\times_S S_i$ are uniquely characterized as subschemes of $S_i\times_S S_j$ by the given property with respect to the coherent sheaf $\mathcal{F}_i|_{S_i\times_S S_j}\cong \mathcal{F}_j|_{S_i\times_S S_j}$ on $\mathscr{X}_{S_i\times_S S_j}$. As it's clear that the cocycle condition on any triple product $S_i\times_S S_j\times_S S_k$ is satisfied, it follows from \cite[\href{https://stacks.math.columbia.edu/tag/0247}{Tag 0247}]{stacks-project} that $S_{\phi(t)}$ exists as a scheme over $S$ (see also \cite[\href{https://stacks.math.columbia.edu/tag/01OX}{Tag 01OX},\href{https://stacks.math.columbia.edu/tag/02JR}{Tag 02JR}]{stacks-project}).
		
		It remains to show that $S_{\phi(t)}$ has property (\textbf{f}). Both the flatness of $\mathcal{F}_T$ and the computation for the reduced Hilbert polynomial $\mathrm{rh}_{\mathcal{F}_T}(t)$ can be checked fppf locally for the cover $S'/S$. The claim follows then from the construction of $S_{\phi(t)}$.
	\end{proof}
	
	For any locally Noetherian $S$-scheme $T$, write $H^{\phi(t)}_{\mathscr{X}/S}(T)$ for the set \begin{equation}\label{eq: func} H^{\phi(t)}_{\mathscr{X}/S}(T):=\left\{V\subset \mathscr{X}_T\middle\vert \begin{array}{c}
			V \text{ is proper and flat over } T\\
			\text{and } \mathrm{rh}_V(t)=\phi(t)
		\end{array}\right\}.
	\end{equation}
	The association of $T$ to $H^{\phi(t)}_{\mathscr{X}/S}(T)$ defines a contravariant functor from the category of locally Noetherian $S$-schemes to the category of sets. For a morphism $\rho:T'\rightarrow T$, the associated map $H^{\phi(t)}_{\mathscr{X}/S}(T)\rightarrow H^{\phi(t)}_{\mathscr{X}/S}(T')$ sends a subscheme $V\subset \mathscr{X}_T$ to $V\times_T T'\subset \mathscr{X}_{T'}$ where the fiber product is taken along the morphism $\rho$.
	
	\begin{thm}\label{thm: twhilb}
		Let $\mathscr{X}/S$ be a Severi--Brauer scheme over a Noetherian base scheme $S$. Then, for every polynomial $\phi(t)\in \mathbb{Q}[t]$, there exists an $S$-scheme $\mathbf{Hilb}_{\phi(t)}^{\mathrm{tw}}(\mathscr{X}/S)$ which represents the functor $H^{\phi(t)}_{\mathscr{X}/S}$ from \emph{(\ref{eq: func})}. 
		
		In particular, there is a subscheme \[\mathbf{Univ}^{\mathrm{tw}}_{\phi(t)}(\mathscr{X}/S)\subset \mathscr{X}\times_S \mathbf{Hilb}_{\phi(t)}^{\mathrm{tw}}(\mathscr{X}/S)\] and, for any locally Noetherian $S$-scheme $T$, there is an equality \[\mathrm{Hom}_S(T,\mathbf{Hilb}_{\phi(t)}^{\mathrm{tw}}(\mathscr{X}/S))=H^{\phi(t)}_{\mathscr{X}/S}(T)\] where a map $f:T\rightarrow \mathbf{Hilb}_{\phi(t)}^{\mathrm{tw}}(\mathscr{X}/S)$ corresponds to the subscheme \[V\cong \mathbf{Univ}^{\mathrm{tw}}_{\phi(t)}(\mathscr{X}/S)\times_{\mathscr{X}\times_S \mathbf{Hilb}_{\phi(t)}^{\mathrm{tw}}(\mathscr{X}/S)} \mathscr{X}\times_S T.\]
	\end{thm}
	
	\begin{proof}
		The proof is essentially the same as in \cite[Theorem I.1.4]{MR1440180}. The only change that needs to be made, taking Lemma \ref{lem: exist} into account, is that one realizes the Hilbert scheme $\mathbf{Hilb}_{\phi(t)}^{tw}(\mathscr{X}/S)$ embedded in the Grassmannian $S$-bundle $\mathscr{Y}=\mathbf{Gr}_S(\phi(N), \pi_*\mathcal{L})$ of rank $\phi(N)$ quotient bundles of the locally free $\pi_*\mathcal{L}$, where $\mathcal{L}=(\det\mathcal{Q})^{\otimes N/(n+1)}$ and $N>0$ is an integer divisible by $n+1$ such that $h^i(V,\mathcal{O}_V(N))=0$ for any subscheme $V\subset \mathbb{P}^n$ with Hilbert polynomial $\phi(t)$.
	\end{proof}
	
	\begin{defn}
		We'll call $\mathbf{Hilb}_{\phi(t)}^{\mathrm{tw}}(\mathscr{X}/S)$ \textit{the Hilbert scheme of $\mathscr{X}/S$} that parameterizes subschemes with reduced Hilbert polynomial $\phi(t)$. The superscript $\mathrm{tw}$ is a reminder that this is a twist of one of the usual Hilbert schemes of a projective bundle as the next remark notes.
	\end{defn}
	
	\begin{rmk}
		If $\mathscr{X}/S$ is split, i.e.\ if $\mathscr{X}/S$ is a projective bundle $\mathbb{P}_S(\mathcal{E})$ for some vector bundle $\mathcal{E}$ on $S$, then the above theorem recovers the usual Hilbert scheme $\mathbf{Hilb}_{\phi(t)}(\mathbb{P}_S(\mathcal{E})/S)$. This also shows the following statement: if $\mathscr{X}/S$ is any Severi--Brauer scheme over a Noetherian base scheme $S$, and if $S'/S$ is an fppf cover splitting $\mathscr{X}/S$, then there are splitting isomorphisms \[\mathbf{Hilb}_{\phi(t)}^{\mathrm{tw}}(\mathscr{X}/S)\times_S S'\cong \mathbf{Hilb}_{\phi(t)}(\mathscr{X}_{S'}/S')\] as claimed in the beginning of this section. Consequently, the scheme $\mathbf{Hilb}_{\phi(t)}^{\mathrm{tw}}(\mathscr{X}/S)$ inherits any property of $\mathbf{Hilb}_{\phi(t)}(\mathscr{X}_{S'}/S')$ that can be checked fppf locally on the base, i.e.\ being finite-type, proper, or smooth over $S$ holds if it also does over $S'$.
	\end{rmk}
	
	\begin{rmk}\label{rmk: compare}
		Given any Severi--Brauer scheme $\mathscr{X}/S$ with structure map $\pi:\mathscr{X}\rightarrow S$, it follows from \cite[\href{https://stacks.math.columbia.edu/tag/01VR}{Tag 01VR}]{stacks-project} that $\mathcal{L}=\det \mathcal{Q}$ is a $\pi$-relatively very ample line bundle. Hence $\pi$ is projective with respect to $\mathcal{L}$ and for any polynomial $\phi(t)\in \mathbb{Q}[t]$ there is a usual Hilbert scheme $\mathbf{Hilb}_{\phi(t)}(\mathscr{X}/S)$ parametrizing flat and proper subschemes of $\mathscr{X}$ whose Hilbert polynomial with respect to $\mathcal{L}$ is $\phi(t)$. If $\mathscr{X}$ has constant relative dimension $n-1$ over $S$, then there is an isomorphism \[\mathbf{Hilb}^{tw}_{\phi(t)}(\mathscr{X}/S)\cong \mathbf{Hilb}_{\phi(nt)}(\mathscr{X}/S)\] where, on the right, $\phi(nt)$ is taken with respect to $\mathcal{L}$.
		
		However, there are some benefits to the construction of $\mathrm{Hilb}_{\phi(t)}^{tw}(\mathscr{X}/S)$. (For example, the twisted and usual Hilbert scheme are both realized as subschemes of certain projective bundles; however, the relative codimension of the twisted Hilbert scheme will always be much lower than that of the usual one under these embeddings). 
	\end{rmk}
	
	The infinitesimal theory of $\mathbf{Hilb}_{\phi(t)}^{\mathrm{tw}}(\mathscr{X}/S)$ can also be checked on an fppf cover of the base, so we get the following corollary using the fact that the scheme $\mathbf{Hilb}_{\phi(t)}^{\mathrm{tw}}(\mathscr{X}/S)$ is fppf locally, e.g.\ on a cover $S'/S$ splitting $\mathscr{X}/S$, isomorphic to $\mathbf{Hilb}_{\phi(t)}(\mathbb{P}_{S'}^n/S')$.
	
	\begin{cor}\label{cor: obs}
		Let $\mathscr{X}/S$ be a Severi--Brauer scheme over $S$. Let $s\in S$ be a point, let $F$ be a field, and let $p:\mathrm{Spec}(F)\rightarrow s$ be a morphism. Let $V\subset \mathscr{X}_F$ be a subscheme with ideal sheaf $\mathcal{I}_V$ and reduced Hilbert polynomial $\mathrm{rh}_V(t)=\phi(t)$. Then the following are true:
		
		\begin{enumerate}[leftmargin=*]
			\item\label{cor: obs1} The Zariski tangent space of $\mathbf{Hilb}_{\phi(t)}^{\mathrm{tw}}(\mathscr{X}_F/F)$ at the $F$-point given by $V$ via Theorem \ref{thm: twhilb} is naturally isomorphic to \[\mathrm{Hom}_{\mathcal{O}_{\mathscr{X}_F}}(\mathcal{I}_V,\mathcal{O}_V)=\mathrm{Hom}_{\mathcal{O}_V}(\mathcal{I}_V/\mathcal{I}_V^2,\mathcal{O}_V).\]
			\item\label{cor: obs2} If $S$ is equidimensional at $s$, then the dimension of every irreducible component of $\mathbf{Hilb}_{\phi(t)}^{\mathrm{tw}}(\mathscr{X}/S)$ at the $F$-point defined by $V$ is at least \[\mathrm{dim}_F \mathrm{Hom}_{\mathcal{O}_{\mathscr{X}_F}}(\mathcal{I}_V,\mathcal{O}_V)-\mathrm{dim}_F \mathrm{Ext}^1_{\mathcal{O}_{\mathscr{X}_F}}(\mathcal{I}_V,\mathcal{O}_V)+\mathrm{dim}_s S.\]
			\item\label{cor: obs3} If both $S$ is equidimensional at $s$ and if $V\subset \mathscr{X}_F$ is (fppf) locally unobstructed, then the dimension of every irreducible component of $\mathbf{Hilb}_{\phi(t)}^{\mathrm{tw}}(\mathscr{X}/S)$ at any point in the image of the point defined by $V$ is at least \[\mathrm{dim}_F \mathrm{Hom}_{\mathcal{O}_V}(\mathcal{I}_V/\mathcal{I}_V^2,\mathcal{O}_V)-\mathrm{dim}_F \HH^1(V,\mathcal{H}om(\mathcal{I}_V/\mathcal{I}_V^2,\mathcal{O}_V))+\mathrm{dim}_s S.\]
		\end{enumerate}
		Moreover, in either of the cases (\ref{cor: obs2}) or (\ref{cor: obs3}) above, if the lower bound given for the dimension is equal to the dimension of every irreducible component of $\mathbf{Hilb}^{\mathrm{tw}}_{\phi(t)}(\mathscr{X}/S)$ at the point defined by $V$, then the map \[\mathbf{Hilb}^{\mathrm{tw}}_{\phi(t)}(\mathscr{X}/S)\rightarrow S\] is a local complete intersection morphism at that point.
	\end{cor}
	
	\begin{proof}
		This is a combination of \cite[Theorems I.2.8, I.2.10, I.2.15]{MR1440180}. See \cite[Definition I.2.11]{MR1440180} for the definition of locally unobstructed subschemes.
	\end{proof}

	\section{Generic geometrically elliptic normal curves}\label{sec: ggenc}
	From now on, we work in the following setting: we fix a base field $k$, a $k$-central simple $k$-algebra $A$, and we let $X=\SB(A)$ be the associated Severi--Brauer variety of $A$. We use the triple $(d,n,m)$ to refer to the degree, index, and exponent of $A$ respectively, i.e.\ \[d=\mathrm{deg}(A),\quad n=\mathrm{ind}(A), \quad m=\mathrm{exp}(A).\]
	Write \begin{equation}\psi_X:\mathbf{Univ}^{\mathrm{tw}}_{\phi(t)}(X/k)\rightarrow \mathbf{Hilb}^{\mathrm{tw}}_{\phi(t)}(X/k)
	\end{equation} for the canonical map coming from the projection. (By slight abuse of notation, we use the same $\psi_X$ regardless of the function $\phi(t)$ under consideration). For each irreducible component $V\subset \mathbf{Hilb}^{\mathrm{tw}}_{\phi(t)}(X/k)$ we let $\eta_V$ denote the generic point of $V$. If $\phi(t)=rt+s$ is linear then, for each such $V$, the generic fiber $\psi_X^{-1}(\eta_V)$ is the union of a curve and a finite number points.
	
	Of particular interest is the following component of $\mathbf{Hilb}^{\mathrm{tw}}_{rt}(X/k)$ for any integer $r\geq 1$ such that $n$ divides $r$.
	
	\begin{defn}
		Let $\mathrm{Ell}_{r}(X)\subset \mathbf{Hilb}^{\mathrm{tw}}_{rt}(X/k)$ denote the union of the irreducible components $V$ of $\mathbf{Hilb}^{\mathrm{tw}}_{rt}(X/k)$ whose generic fiber $\psi_X^{-1}(\eta_V)$ is a smooth and geometrically connected curve of genus $1$.
	\end{defn}
	
	If either $\mathrm{dim}(X)=2$ and $r=3$, or if $\mathrm{dim}(X)\geq 3$ and $r\geq 3$ is an arbitrary integer, then the scheme $\mathrm{Ell}_r(X)$ is nonempty. Indeed, this can be checked over an algebraic closure $\overline{k}$ of $k$. It then follows from the two facts: \begin{enumerate}[leftmargin=*]
\item a smooth cubic curve in $\mathbb{P}^2_{\overline{k}}$ is an elliptic curve, and
\item $\mathbb{P}^3_{\overline{k}}$ (and so also $\mathbb{P}^n$ for $n\geq 3$) contains smooth elliptic curves of every degree $r\geq 3$.
\end{enumerate}
	
	\begin{prop}\label{prop: elln}
		Suppose $A$ is a central simple $k$-algebra of degree $d$ and of index $n$. Then the following are true:
		\begin{enumerate}[leftmargin=*]
			\item\label{thm: ellnirr} $\mathrm{Ell}_d(X)$ is geometrically irreducible with $\mathrm{dim}(\mathrm{Ell}_d(X))=d^2$;
			\item\label{thm: ellnmax} if $A$ has division and either $A$ is cyclic or, if $A$ contains a maximal subfield $F\subset A$ whose Galois closure $E/k$ is a Galois extension of degree $2n$ with dihedral Galois group, then $\mathrm{Ell}_n(X)(k)\neq\emptyset$.
		\end{enumerate}
	\end{prop}
	
	\begin{proof}
		We first prove (\ref{thm: ellnmax}). In either case, let $x$ be a point of $X$ with $k(x)$ either a cyclic Galois extension $E/k$ of $k$ of degree $n$ (in the first case) or a maximal subfield $k(x)\subset A$ with Galois closure $E/k$ a dihedral Galois extension of degree $2n$ (in the second case). The field $E$ splits $X$ and $k(x)\otimes_k E\cong E^{\oplus n}$ either way. Let $H\subset \mathrm{Gal}(E/k)$ be a cyclic subgroup of order $n$. Pick an $E$-rational point $p$ in $x_E$ and let $L$ be the line through $p$ and $gp$ for any generator $g$ of $H$. 
		
		The union of the $H$-translates of $L$ forms a $\mathrm{Gal}(E/k)$-orbit which descends to a scheme $V\subset X$ defined over $k$. Geometrically, the scheme $V_{\overline{k}}$ is an $n$-gon of lines through the points $x_{\overline{k}}$. Hence $\mathrm{rh}_V(t)=nt$. We claim the point defined by $V$ in $\mathbf{Hilb}_{nt}^{\mathrm{tw}}(X/k)$ is contained in $\mathrm{Ell}_n(X)$. 
		
		Actually, as $V_{\overline{k}}$ is the scheme-theoretic union of lines we can use the exact sequence \cite[\href{https://stacks.math.columbia.edu/tag/0C4J}{Tag 0C4J}]{stacks-project} \begin{equation}\label{eq: union} 0\rightarrow \mathcal{O}_{C\cup D}\rightarrow \mathcal{O}_C\oplus \mathcal{O}_D\rightarrow \mathcal{O}_{C\cap D}\rightarrow 0
		\end{equation} where $V_{\overline{k}}=C\cup D$, with $C$ a chain of $n-1$ lines and $D$ a line closing the $n$-gon, to compute that $h^1(V,\mathcal{O}_V)=1$ and that $h^1(V_{\overline{k}},\mathcal{O}_{V_{\overline{k}}}(1))=0$ by tensoring the exact sequence with $\mathcal{O}_{X_{\overline{k}}}(1)$. Since $V_{\overline{k}}$ has lci singularities, one can apply \cite[Proposition 29.9]{MR2583634} to find that $V_{\overline{k}}$ is smoothable.
		
		More precisely, we find that $\mathbf{Hilb}^{\mathrm{tw}}_{nt}(X/k)$ is smooth at the $k$-rational point defined by $V\subset X$ and, over an algebraic closure, there is an integral curve passing through both the point corresponding to $V_{\overline{k}}\subset X_{\overline{k}}$ and the subset of $\mathrm{Ell}_n(X_{\overline{k}})$ parametrizing smooth and connected curves. In particular, the embedding $V\subset X$ defines a point of $\mathrm{Ell}_n(X)(k)$ completing the proof of (\ref{thm: ellnmax}).	
		
		Now we prove (\ref{thm: ellnirr}). If $d=3$, then $\mathbf{Hilb}^{\mathrm{tw}}_{3t}(X/k)$ is isomorphic to $\mathbb{P}^9$.
		So we can assume $d>3$. Then $\mathrm{Ell}_d(X)$ is geometrically irreducible by \cite[Theorem 8]{MR875083}. The dimension of $\mathrm{Ell}_d(X)$ can also be determined geometrically. Essentially, if $C\subset X_{\overline{k}}$ is smooth of degree $d$ and genus 1 then one can compute \[h^0(C,\mathcal{N}_{C/X_{\overline{k}}})=d^2 \quad \mbox{and}\quad h^1(C,\mathcal{N}_{C/X_{\overline{k}}})=0\] using the normal bundle sequence (and the Euler sequence for $X_{\overline{k}}$). This shows both that $\mathrm{dim}(\mathrm{Ell}_d(X))\leq d^2$, from Corollary \ref{cor: obs} (\ref{cor: obs1}), and that $\mathrm{dim}(\mathrm{Ell}_d(X))\geq d^2$, from Corollary \ref{cor: obs} (\ref{cor: obs3}).
	\end{proof}
	
	\begin{rmk}\label{rmk: star}
		The proof of (\ref{thm: ellnmax}) in Proposition \ref{prop: elln} above is an extension of an argument due to Jason Starr, cf.\ \cite{286458}. There Starr's goal is to use the fact that $V$ defines a smooth $k$-rational point on $\mathbf{Hilb}^{\mathrm{tw}}_{nt}(X/k)$ to construct a smooth genus 1 curve on any Severi--Brauer variety $X$ defined over a large (also called ample) field $k$ (e.g.\ a $p$-special field or the fraction field of a Henselian DVR).
		
		We can elaborate on Starr's result in the setting of Proposition \ref{prop: elln}, i.e.\ when $A$ is a division $k$-algebra satisfying the assumptions of (\ref{thm: ellnmax}). Indeed, the scheme $\mathrm{Ell}_n(X)$ is projective so we can construct a smooth curve $E$ with a $k$-rational point mapping to the $k$-point $x$ associated to the $n$-gon $V$ constructed in the proof of Proposition \ref{prop: elln} (\ref{thm: ellnmax}) as follows.
		
		Let $y$ be any point of $\mathrm{Ell}_n(X)$ whose associated subscheme $C\subset X$ is a smooth geometrically connected curve of genus 1. Let $I=\{x,y\}$. Consider the blowup $\mathrm{Bl}_I(\mathrm{Ell}_n(X))$ with center the points $I\subset \mathrm{Ell}_n(X)$. Since $\mathrm{Ell}_n(X)$ is projective, there is some embedding of the blowup $\mathrm{Bl}_I(\mathrm{Ell}_n(X))\subset \mathbb{P}^{M}$. A general linear section of the correct codimension intersects $\mathrm{Bl}_I(\mathrm{Ell}_n(X))$ in a curve (smooth near $x$) by Bertini's theorem \cite[Th\'eor\`eme 6.10 et Corollaire 6.11]{MR725671}. A general section of the same codimension intersects the exceptional divisor $\mathbb{P}^{n^2-1}\subset \mathrm{Bl}_I(\mathrm{Ell}_n(X))$ over $x$ in a $k$-rational point and the exceptional divisor over $y$ in some number of points. So we can choose a section $E'\subset \mathrm{Bl}_I(\mathrm{Ell}_n(X))$ doing all three things at once. The normalization $E$ of $E'$ is a curve with all the stated properties.
		
		Over a large (also called ample) field $k$, any irreducible curve having a smooth $k$-rational point has infinitely many $k$-rational points. Thus the curve $E$ has infinitely many $k$-rational points and the image along the composition of the normalization and blowdown \[E\rightarrow E'\rightarrow  \mathrm{Bl}_I(\mathrm{Ell}_n(X))\rightarrow \mathrm{Ell}_n(X)\] has nontrivial intersection with the open subset of $\mathrm{Ell}_n(X)$ consisting of smooth and geometrically connected genus 1 curves.
	\end{rmk}
	
	\begin{exmp}\label{exmp: gen}
		If $A$ is a cyclic division $k$-algebra of index $n$, there are lots of field extensions $F/k$ where $X_F$ contains a smooth geometrically connected curve of genus 1 and where the algebra $A_F$ has index $n$. When $n=p^r$ is a power of a prime $p$, Remark \ref{rmk: star} shows this holds for a minimal $p$-special field $F/k$ contained in an algebraic closure $\overline{k}/k$. When the index $n$ is arbitrary one can instead use the field $k((t))$, which is the fraction field of a Henselian DVR, and apply Remark \ref{rmk: star}. The index remains $n$ here since $A_{k((t))}$ specializes to $A$ (Lemma \ref{lem: spec}).
		
		One can also construct ``generic" examples for an arbitrary division algebra $A$ of index $n$ as follows. First, one can use \cite{MR2533621} to find a field extension $F/k$ with $A_F$ cyclic of index $n$ and such that the restriction map $\mathrm{Br}(k)\rightarrow \mathrm{Br}(F)$ is an injection. Setting $L=F(\mathrm{Ell}_n(X_F))$ and noting that $\mathrm{Ell}_n(X_F)$ has a smooth $F$-rational point corresponding to the $n$-gon above, we have that the restriction $\mathrm{Br}(F)\rightarrow \mathrm{Br}(L)$ is an injection by \cite[Lemma 5.4.7]{MR3727161}. Now Lemma \ref{lem: spec} below shows that $A_L$ remains index $n$ since $A_L$ specializes to $A_F$. Hence also:
\begin{enumerate}[leftmargin=*]
\item the extension of $A$ to $E=k(\mathrm{Ell}_n(X))$ has index $n$, 
\item the restriction map $\mathrm{Br}(k)\rightarrow \mathrm{Br}(E)$ is injective,
\item and $X_E$ contains the ``generic" smooth and geometrically connected curve of genus $1$ of $X$.
\end{enumerate}
	\end{exmp}
	
	\begin{exmp}\label{exmp: genfam}
		Let $n\geq 3$ be an integer and fix a divisor $m\geq 1$ of $n$. Set $G=\mathrm{SL}_n/\mu_m$ to be the quotient of the special linear group by the sub-group scheme of $m$th roots of unity. Fix a faithful representation $G\rightarrow \mathrm{GL}_N$ for some $N\gg0$ and let $\pi:\mathrm{GL}_N\rightarrow \mathrm{GL}_N/G$ be the quotient. If $P\subset G$ is a parabolic subgroup such that $P\backslash G\cong \mathbb{P}^{n-1}$, then $\pi$ is equivariant for the right-action of $P$ and the quotient by this action yields a Severi--Brauer scheme $\pi_0:P\backslash \mathrm{GL}_N\rightarrow \mathrm{GL}_N/G$. One can therefore consider the relative $\mathrm{GL}_N/G$-scheme $\mathbf{Hilb}_{nt}^{tw}(\pi_0)$ and, if $\eta$ is the generic point of the (smooth and geometrically irreducible) scheme $\mathrm{GL}_N/G$, we can define the relative $\mathrm{GL}_N/G$-scheme $\mathrm{Ell}_n(\pi_0)$ as the scheme theoretic closure of $\mathrm{Ell}_n(\pi_0\times_{\mathrm{GL}_N/G} \eta)$ inside $\mathbf{Hilb}_{nt}^{tw}(\pi_0)$.
		
		The scheme $\mathrm{Ell}_n(\pi_0)$ is proper and surjective over $\mathrm{GL}_N/G$ and, for any field extension $F/k$ and for any $F$-point $x\in(\mathrm{GL}_N/G)(F)$, the fiber $\mathrm{Ell}_n(\pi_0)\times_{\mathrm{GL}_N/G} x$ contains $\mathrm{Ell}_n(\pi_0\times_{\mathrm{GL}_N/G} x)$ as a closed subscheme. By \cite[\href{https://stacks.math.columbia.edu/tag/0559}{Tag 0559}]{stacks-project}, there is then an open subscheme $W\subset \mathrm{GL}_N/G$ such that for any $x\in W(F)$ there is an equality \[\mathrm{Ell}_n(\pi_0)\times_{\mathrm{GL}_N/G}x \cong \mathrm{Ell}_n(\pi_0\times_{\mathrm{GL}_N/G} x).\]
		
		If the base field $k$ is infinite, then the relative Severi--Brauer scheme $\pi_0$ is versal (cf.\ \cite[Ch.\ 1 \S5]{MR1999383}) in the sense that for any nonempty open subscheme $U\subset \mathrm{GL}_N/G$, for any field extension $F/k$, and for any Severi--Brauer variety $X$ associated to an $F$-central simple $F$-algebra $A$ with $\mathrm{deg}(A)=n$ and $\mathrm{exp}(A)$ dividing $m$, there exists an $F$-point $x\in U(F)$ so that $X\cong \pi_0^{-1}(U) \times_U x$. The scheme $\mathrm{Ell}_n(\pi_0)\times_{\mathrm{GL}_N/G} W$ and its universal family, considered over $W$, is similarly versal for geometrically elliptic normal curves on Severi--Brauer varieties.
		
		Moreover, using Example \ref{exmp: gen}, there exists a \textit{generic geometrically elliptic normal curve} $C^{gen}_{n,m}$ on the base extension $(X^{gen}_{n,m})_E$ of the generic Severi--Brauer variety $X^{gen}_{n,m}=\pi_0^{-1}(\eta)$, where $E$ is the function field of the scheme $\mathrm{Ell}_n(\pi_0\times_{\mathrm{GL}_N/G} \eta)$. Fix any field $F/k$, fix a point $x\in W(F)$ corresponding to a Severi--Brauer variety $X$, and fix a geometrically elliptic normal curve $C\subset X$. The point $s_x$ in $S=\mathrm{Ell}_n(\pi_0\times_{\mathrm{GL}_N/G} x)$ associated to the subscheme $C\subset X$ is geometrically regular. Hence there exists a sequence of DVRs $(R_0,\mathfrak{m}_0),...,(R_{j(s_x)},\mathfrak{m}_{j(s_x)})$ satisfying the following conditions:
		\begin{enumerate}[leftmargin=*]
			\item $\mathrm{Frac}(R_0)=F(\mathrm{Ell}_n(\pi_0)|_W\times_k F)=F(\mathrm{Ell}_n(\pi_0\times_k F)):=E'$,  
			\item $R_i/\mathfrak{m}_i\cong \mathrm{Frac}(R_{i+1})$
			\item $R_{j(s_x)}/\mathfrak{m}_{j(s_x)}\cong F(s_x)$.
		\end{enumerate} There are also smooth $\mathrm{Spec}(R_i)$-schemes, gotten by base change of the universal family, which at one end gives $C_{n,m}^{gen}\times_E E'$ and the other $C$. In this way the generic geometrically elliptic normal curve specializes to any other geometrically elliptic normal curve in any Severi--Brauer variety defined over any field extension of $k$.
	\end{exmp}
	
	Recall that the period $\mathrm{per}(C)$ of a smooth, proper, and geometrically integral curve $C/k$ is the smallest integer $m\geq 1$ so that $\mathbf{Pic}^m_{C/k}(k)\neq \emptyset$. Equivalently, the period of $C/k$ is the order of the element $[\mathbf{Pic}^1_{C/k}]$ inside the first Galois cohomology group $\mathrm{H}^1(k, \mathbf{Pic}^0_{C/k})$. We frequently use the fact that a $k$-rational point of $\mathbf{Pic}^m_{C/k}$ corresponding to a globally generated line bundle on $C_{\overline{k}}$ determines a morphism to a Severi--Brauer variety (and conversely), cf. \cite[Theorem 1.1]{MR3644253}.
	
	Recall also that the index $\mathrm{ind}(C)$ of $C$ is the unique positive integer generating the image of the degree map $\mathrm{deg}:\CH_0(C)\rightarrow \mathbb{Z}$. We have that $\mathrm{per}(C)$ divides $\mathrm{ind}(C)$ and if the genus of $C$ satisfies $g(C)=1$, then  $\mathrm{ind}(C)$ divides $\mathrm{per}(C)^2$, see \cite[Theorem 8]{MR0242831}. In the following theorem we keep the notation of Example \ref{exmp: genfam} (in particular, the base field $k$ is assumed to be infinite).
	
	\begin{thm}\label{thm: perind}
		Let $n\geq 3$ be an integer, and let $m> 1$ be a divisor of $n$ such that $n$ and $m$ have the same prime factors (i.e.\ $m\mid n\mid m^\infty$). Assume, additionally, that $n$ is not divisible by the characteristic of $k$.
		
		Then the generic geometrically elliptic normal curve $C^{gen}_{n,m}$ above has index $\mathrm{ind}(C^{gen}_{n,m})=nm$ and $\mathrm{per}(C^{gen}_{n,m})=n$.
	\end{thm}
	
	\begin{rmk}
		Let $A^{gen}_{n,m}$ be the central simple $k(\eta)$-algebra associated to the generic Severi--Brauer variety $X^{gen}_{n,m}$. If $n=st$ is a factorization by integers $s$ and $t$ such that $\mathrm{gcd}(t,m)=1$ and $s$ and $m$ share the same prime factors, then \[\mathrm{deg}(A^{gen}_{n,m})=n,\quad \mathrm{ind}(A^{gen}_{n,m})=s,\quad \mbox{and} \quad \mathrm{exp}(A^{gen}_{n,m})=m.\]
		So the assumptions on $n$ and $m$ in Theorem \ref{thm: perind} describe, equivalently, exactly those cases where $A^{gen}_{n,m}$ is a division algebra.
	\end{rmk}
	
	\begin{proof}
		We first deal with the case when $n=m$. Since $C^{gen}_{n,n}$ embeds as a geometrically elliptic normal curve in a Severi--Brauer variety of dimension $n-1$, we find $\mathrm{per}(C^{gen}_{n,n})$ divides $n$. To prove $\mathrm{per}(C^{gen}_{n,n})=n$, it therefore suffices to show $\mathrm{ind}(C^{gen}_{n,n})=n^2$.
		
		Because of our assumption that the characteristic of the base field $k$ does not divide $n$, we can find a field extension $F/k$ and a smooth, proper, and geometrically integral $F$-curve $C$ of genus $g(C)=1$ with $\mathrm{ind}(C)=n^2$ and $\mathrm{per}(C)=n$. (By the remarks at the end of \S4 in \cite{MR0106226}, one can take $F=\overline{k}((t_1))((t_2))$ for an algebraic closure $\overline{k}$ of $k$). After base extension from $k$ to $F$, it follows that $(C^{gen}_{n,n})_{E'}$ specializes (along a sequence of DVRs) to $C$ as above; here $E'$ is the function field of $\mathrm{Ell}_n(\pi_0\times_k F)$ as before. Hence, by \cite[Proposition 20.3 (a)]{MR1644323}, the index $\mathrm{ind}(C^{gen}_{n,n})$ is divisible by $n^2$ which implies that it actually is $n^2$.
		
		When $n\neq m$, we can similarly argue by specialization. In this case, we still have $\mathrm{per}(C^{gen}_{n,m})$ divides $n$ since $C^{gen}_{n,m}$ embeds in a Severi--Brauer variety of dimension $n-1$ as a geometrically elliptic normal curve. Since $n$ is indivisible by the characteristic of $k$, we can construct (see Lemma \ref{lem:eqperind} below) a smooth, proper, and geometrically integral curve $C$ over a field extension $F/k$ with $\mathrm{per}(C)=\mathrm{ind}(C)=n$. This curve $C$ embeds as a geometrically elliptic normal curve on the trivial Severi--Brauer variety $\mathbb{P}^{n-1}_F$, which is associated to a central simple $F$-algebra of degree $n$ and exponent exactly equalling $1$. Hence we can specialize $(C^{gen}_{n,m})_{E'}$ to $C$ along a sequence of DVRs; here $E'=F(\mathrm{Ell}_n(\pi_0\times_k F))$. 

As each relative curve that appears over a DVR in this process is projective, smooth, and has geometrically integral fibers, we can form their associated relative Picard schemes \cite[Theorem 4.8]{MR2223410}. In this way we can also specialize from $\mathbf{Pic}^d_{(C^{gen}_{n,m})_{E'}/{E'}}\cong \mathbf{Pic}^d_{C^{gen}_{n,m}/E}\times_E E'$, where we recall $E$ is the function field of $\mathrm{Ell}_n(\pi_0\times_{\mathrm{GL}_N/G} \eta)$, to $\mathbf{Pic}^d_{C/F}$, for each integer $d$ dividing $n$, along a sequence of DVRs. Since the period can only decrease when extending the base field, we can apply \cite[Proposition 20.3 (a)]{MR1644323} to show that $\mathrm{per}(C^{gen}_{n,m})=n$ as claimed.
		
		To compute the index of $C^{gen}_{n,m}$, we also use a specialization argument. Let $A$ be the central simple $E$-algebra corresponding to $(X^{gen}_{n,n})_E$ and let $X=\SB(A^{\otimes m})$. Since $(C^{gen}_{n,n})_{E(X)}$ sits on the Severi--Brauer variety $(X^{gen}_{n,n})_{E(X)}$, which is associated to the division algebra $A_{E(X)}$ of index $n$ and exponent $m$ by \cite[Theorem 2.1]{MR1061778}, we can specialize $(C^{gen}_{n,m})_{E'}$ to this curve along a sequence of DVRs; $E'=E(X)(\mathrm{Ell}_n(\pi_0\times_k E(X)))$. We show in Lemma \ref{lem: indcurve} below that the curve $(C^{gen}_{n,n})_{E(X)}$ has index $nm$. Thus, using \cite[Proposition 20.3 (a)]{MR1644323} again, we get $\mathrm{ind}(C^{gen}_{n,m})\geq nm$. However, it's possible to see that we must also have $\mathrm{ind}(C^{gen}_{n,m})\leq nm$ as we now explain.
		
		Indeed, if $B$ is the central simple $k(\eta)$-algebra associated to $X^{gen}_{n,m}$ then $B$ has index $n$ and exponent $m$. If $E=k(\mathrm{Ell}_n(\pi_0))$ is the given function field, then $(X^{gen}_{n,m})_E$ is associated to the algebra $B_E$ which still has index $\mathrm{ind}(B_E)=n$ and exponent $\mathrm{exp}(B_E)=m$ as the restriction $\mathrm{Br}(k(\eta))\rightarrow \mathrm{Br}(E)$ is an injection  (see Example \ref{exmp: gen}). If $H$ is a general divisor of $(X^{gen}_{n,m})_E$ of degree $\mathrm{exp}(B_E)=m$ then \[[C^{gen}_{n,m}\cap H]=[C^{gen}_{n,m}][H]=m[p]\] holds in $\CH_0((X^{gen}_{n,m})_E)$ for some point $p$ of degree $\mathrm{ind}(B_E)=n$; this is a direct computation, cf.\ \cite[Corollary 7.3]{MR2278759}. Now the left hand side of this equation has degree some multiple of the index of $C^{gen}_{n,m}$ whereas the right hand side has degree $nm$.\end{proof}
	
	We needed two lemmas for the above proof. The first of these lemmas constructs curves of equal period and index over an extension of $k$. The proof below is adapted from \cite{286395}.
	
	\begin{lem}\label{lem:eqperind}
		Let $n\geq 1$ be an integer not divisible by the characteristic of $k$. Let $\overline{k}$ be a fixed algebraic closure of $k$. Write $F=\overline{k}((t))$ for the field of formal Laurent series in $t$ over $\overline{k}$. Then there exists a smooth and proper genus one curve $C/F$ with $\mathrm{per}(C)=\mathrm{ind}(C)=n$.
	\end{lem}
	
	\begin{proof}
		Let $E/k$ be any elliptic curve. We claim that there exists an element $x\in \mathrm{H}^1(F,E_F)$ having exact order $n$. Using the correspondence between this Galois cohomology group and the Weil--Ch\^{a}telet group for $E_F$, the element $x$ corresponds to an $E_F$-torsor $C/F$ having period $n$. By \cite[Theorem 1]{MR0237506}, the curve $C$ also has index $n$.
		
		The Kummer sequence associated to the multiplication-by-$n$ map on $E_F$ yields the exact sequence \begin{equation}\label{eq:kumm} 0\rightarrow E_F(F)/nE_F(F)\rightarrow \mathrm{H}^1(F,E_F[n])\rightarrow\mathrm{H}^1(F,E_F)[n]\rightarrow 0\end{equation} where $E_F[n]$ is the subgroup scheme of $n$-torsion points of $E_F$. Since $n$ is not divisible by the characteristic of $k$, and since $E$ is defined over $k$, there exists an isomorphism of group schemes $E_F[n]\cong (\mathbb{Z}/n\mathbb{Z})^{\oplus 2}$. Since $F$ admits a cyclic Galois extension of degree $n$ (i.e.\ $F(t^{1/n})$), there exists an element $z\in\mathrm{H}^1(F,E_F[n])$ of exact order $n$.
		
		We claim that the group $E_F(F)/nE_F(F)=0$ so that, by (\ref{eq:kumm}), there exists an element $x$ of order $n$ as desired (the image of $z$, for example). This result seems to be well-known to experts, cf.\ \cite[Remark I.3.6]{MR881804}, but we include a proof here for completeness.

 Let $R=\overline{k}[[t]]$. The restriction $E_R(R)\rightarrow E_F(F)$ is an isomorphism due to the valuative criterion for properness, so it suffices to show that $E_R(R)$ is $n$-divisible. Since we have that $R=\varprojlim_m R/(t^m)$, there is an isomorphism \[\varprojlim_{m}E_R(R/(t^m)) \cong E_R(R).\] We'll show that each $E_{R}(R/(t^m))$ is $n$-divisible by induction on $m$. When $m=1$, the group $E_R(\overline{k})=E_{\overline{k}}(\overline{k})$ is $n$-divisible as $E$ is an elliptic curve. Now assume $E_R(R/(t^m))$ is $n$-divisible for some $m\geq 1$. From restriction we get an exact sequence \begin{equation}\label{eq: formsmooth}0\rightarrow V\rightarrow E_R(R/(t^{m+1}))\rightarrow E_R(R/(t^m))\rightarrow 0\end{equation} with surjectivity on the right from formal smoothness. Here the kernel $V$ is a $\overline{k}$-vector space which is $n$-divisible since the characteristic of $k$ doesn't divide $n$. It follows that $E_R(R/(t^{m+1}))$ is $n$-divisible as well.

Now tensoring (\ref{eq: formsmooth}) with $\mathbb{Z}/n\mathbb{Z}$ and considering the associated long exact sequence of $\mathrm{Tor}^{\mathbb{Z}}_*(\mathbb{Z}/n\mathbb{Z},-)$ shows that the canonical inverse system $\{E_R(R/(t^m))[n]\}_m$ of $n$-torsion subgroups is everywhere surjective. Taking limits of the inverse systems of the exact sequences \[0\rightarrow E_R(R/(t^m))[n]\rightarrow E_R(R/(t^m))\xrightarrow{\cdot n} E_R(R/(t^m))\rightarrow 0\] now shows that $E_R(R)$ is $n$-divisible, since the system $\{E_R(R/(t^m))[n]\}_m$ is Mittag-Leffler, cf.\ \cite[\href{https://stacks.math.columbia.edu/tag/0598}{Tag 0598}]{stacks-project}.
	\end{proof}
	
	The second lemma provides an index reduction formula for the generic curve $C_{n,n}^{gen}$.
	
	\begin{lem}\label{lem: indcurve}
		Let $n\geq 3$ be an integer not divisible by the characteristic of the base field $k$ and fix a divisor $m\geq 1$ of $n$ sharing the same prime factors as $n$ if $m>1$. Let $A$ be the central simple $E$-algebra associated to the Severi--Brauer variety $(X^{gen}_{n,n})_E$. Let $X=\SB(A^{\otimes m})$. 
		
		Then the generic geometrically elliptic normal curve $C^{gen}_{n,n}\subset (X_{n,n}^{gen})_E$ satisfies $\mathrm{ind}\left((C^{gen}_{n,n})_{E(X)}\right)=nm$. Moreover, if $n/m$ is squarefree, then the period of $(C^{gen}_{n,n})_{E(X)}$ is $\mathrm{per}\left((C^{gen}_{n,n})_{E(X)} \right)=n$.
	\end{lem}
	
	\begin{proof}
		Let $C=C^{gen}_{n,n}$ for the proof. Now there exists an exact sequence \begin{equation}\label{eq: P1} 0\rightarrow \mathrm{Pic}(C\times X)\rightarrow \mathbf{Pic}_{C\times X/E}(E)\xrightarrow{\delta} \mathrm{Br}(E)\end{equation} which can be obtained in multiple ways, see for example \cite[Proof of Theorem 2.1]{MR3009747} or \cite[Remark 2.11]{MR2223410}. Important for us are the facts that there is an equality \[\mathbf{Pic}_{C\times X/E}(E)=\mathrm{Pic}((C\times X)_{E^s})^{\mathrm{Gal}(E^s/E)},\] where $E^s$ is a separable closure of $E$, and that there is a geometric realization of the rightmost map of (\ref{eq: P1}), see e.g.\ \cite[Theorem 3.4]{MR3644253}. 
		
		Using the above equality, we can compute $\mathbf{Pic}_{C\times X/E}(E)$ explicitly. There is an exact sequence of $\mathrm{Gal}(E^s/E)$-modules \[0\rightarrow \mathrm{Pic}(C_{E^s})\times  \mathrm{Pic}(X_{E^s})\rightarrow \mathrm{Pic}((C\times X)_{E^s}) \rightarrow H\rightarrow 0\] where $H=\mathrm{Hom}(\left(\mathbf{Pic}^0_{X_{E^s}/E^s}\right)^\vee, \mathbf{Pic}^0_{C_{E^s}/{E^s}})$ and the leftmost nonzero map is the pullback along the two projections, see \cite[\S5.7.1]{MR4304038}. Note that, as $X$ is a Severi--Brauer variety, we have $H=0$ so that \[\mathbf{Pic}_{C\times X/E}(E)\cong (\mathrm{Pic}(C_{E^s})\times  \mathrm{Pic}(X_{E^s}))^{\mathrm{Gal}(E^s/E)} \cong \mathbf{Pic}_{C/E}(E)\times \mathbb{Z}\] where a generator in the second component is given by the class of the line bundle $\mathcal{O}(1)$ on $X_{E^s}\cong \mathbb{P}_{E^s}^{n-1}$. Note that $\mathcal{O}(1)$ maps to $[A^{\otimes m}]$ in the Brauer group $\mathrm{Br}(E)$ under the map $\delta$.
		
		Suppose an $E$-rational point in $\mathbf{Pic}_{C\times X/E}(E)$ is given geometrically by the pair $x=\left(\mathcal{L}, \mathcal{O}(-\ell) \right)$. Then $x$ comes from a line bundle on $C\times X$ only if its image in $\mathrm{Br}(E)$ is trivial, i.e.\ if there is an equality \[0=\delta(\mathcal{L})-[A^{\otimes m\ell}].\] Since $C$ has an $E$-rational divisor of degree $n^2$, whose image in $\mathrm{Br}(E)$ is trivial, we can translate such an $x$ to a pair where the first component has degree $0<d=\mathrm{deg}(\mathcal{L})\leq n^2$. We can even assume $d>1$ since $C\cong \mathbf{Pic}^1_{C/E}$ has no $E$-rational points.
		
		Since $C$ has genus $g(C)=1$, any line bundle representing the point $\mathcal{L}$ on $\mathbf{Pic}_{C/E}^d(E)$ is globally generated and thus defines a morphism \[\varphi:C\rightarrow P\] where $P$ is a Severi--Brauer variety with class $[P]=\delta(\mathcal{L})=[A^{\otimes m\ell}]$ in $\mathrm{Br}(E)$ by \cite[Theorem 3.4]{MR3644253}. For a general Weil divisor $D$ on $P$ of degree $e=\mathrm{exp}(A^{\otimes m \ell})$, the zero-cycle $[\varphi^{-1}(D)]\in \mathrm{CH}_0(C)$ then has degree $de$ considered as a Weil divisor of $C$. Since $\mathrm{ind}(C)=n^2$ by the first two paragraphs of Theorem \ref{thm: perind}, we must have $n^2$ divides $de$. Since $\mathrm{exp}(A^{\otimes m\ell})$ divides $n/m$, we get \[n^2 \mid de \mid d\left(\frac{n}{m}\right).\] Hence $nm$ divides $d$.
		
		Now there is a commutative box (all faces commute)
		\[\begin{tikzcd}[row sep=scriptsize, column sep=scriptsize]
			& \mathrm{Pic}(C\times X) \arrow[dl, twoheadrightarrow] \arrow[rr] \arrow[dd] & & \mathrm{Pic}((C\times X)_{E^s}) \arrow[dl, twoheadrightarrow] \arrow[dd] \\
			\mathrm{Pic}(C_{E(X)})\arrow [rr, crossing over] \arrow[dd] & & \mathrm{Pic}((C_{E^s})_{E^s(X_{E^s})}) \\
			& \CH^0(X) \arrow[dl, equals] \arrow[rr, equals] & & \CH^0(X_{E^s}) \arrow[dl, equals] \\
			\mathbb{Z} \arrow[rr, equals] & & \mathbb{Z} \arrow[from=uu, crossing over]\\
		\end{tikzcd}\]
		where all vertical arrows are pushforward morphisms, all other arrows are pullbacks, and we've identified \[\CH_0(\mathrm{Spec}(E(X)))=\mathbb{Z}=\CH_0(\mathrm{Spec}(E^s(X_{E^s}))),\] see \cite[Proposition 1.7]{MR1644323}. By localization \cite[Corollary 57.11]{MR2427530}, all slanted arrows are surjective and the bottom square is trivially all isomorphisms.
		
		If $\mathcal{L}_0$ is a line bundle on $C_{E(X)}$ we can therefore lift it to a line bundle on $C\times X$ which, over the separable closure $E^s/E$ is of the form $x=(\mathcal{L},\mathcal{O}(-\ell))$ for a line bundle $\mathcal{L}$ on $C_{E^s}$ with $\mathrm{deg}(\mathcal{L})$ a multiple of $nm$ and for some $\ell\in \mathbb{Z}$. By pushing forward to $X_{E^s}$ and restricting to $E^s(X_{E^s})$, we see that $\mathrm{deg}(\mathcal{L}_0)=\mathrm{deg}(\mathcal{L})$ is a multiple of $nm$. Conversely, taking $m$-times the point of $\mathbf{Pic}^n_{C/E}$ corresponding to the embedding $C\subset X^{gen}_{n,n}$ defines a degree $nm$ line bundle on $C_{E(X)}$. Hence $\mathrm{ind}(C_{E(X)})=nm$.
		
		Finally, assume that $n/m$ is squarefree. If the period of $C_{E(X)}$ is $d$, then $d$ divides $n$ since the period of $C$ was $n$. If $d\neq n$, then there is a prime $p$ so that $v_p(d)\leq v_p(n)-1$ where $v_p$ is the $p$-adic valuation. Now we have $d$ divides $nm$ which divides $d^2$ by the period/index relations. However, if $n/m$ is squarefree then $v_p(m)\geq v_p(n)-1$ so that \[v_p(nm)=v_p(n)+v_p(m)\geq 2v_p(n)-1\] while $v_p(d^2)=2v_p(d)\leq 2v_p(n)-2$. Hence $d=n$.
	\end{proof}
	
	\appendix
	\section{On Azumaya algebras}
	\begin{lem}\label{lem: spec}
		Let $R$ be a Noetherian regular local ring with maximal ideal $\mathfrak{m}$, residue field $k=R/\mathfrak{m}$, and fraction field $F$. Suppose that $A$ is an Azumaya $R$-algebra. Then there is an inequality $\mathrm{ind}(A_k)\leq\mathrm{ind}(A_F)$.
	\end{lem}
	
	\begin{proof}
		We consider the $R$-schemes $X_m=\mathbf{SB}_m(A)$, of left (or right) ideal summands of $A$ of reduced dimension $m$, which are \'etale forms of the Grassmannian $R$-schemes $\mathbf{Gr}_R(m,n)$, where $n$ is the square root of the rank of $A$, and for varying $m$. The $F$ and $k$ fibers of the structure map over $R$ are canonically \[\mathbf{SB}_m(A_F)\cong \mathbf{SB}_m(A)\times_R F \quad\mbox{and}\quad \mathbf{SB}_m(A_k)\cong\mathbf{SB}_m(A)\times_R k,\] which have an $F$-rational point, or a $k$-rational point respectively, if and only if the index $\mathrm{ind}(A_F)$, or $\mathrm{ind}(A_k)$ respectively, divides $m$ \cite[Proposition 3]{doi:10.1080/00927879108824131}. We'll show that the assumption $R$ is regular guarantees that $\mathbf{SB}_m(A_k)(k)\neq \emptyset$ whenever $\mathbf{SB}_m(A_F)(F)\neq \emptyset$.
		
		For this, we first note that $R$ admits a sequence of discrete valuation rings $R_0,...,R_t$ with maximal ideals $\mathfrak{m}_0,...,\mathfrak{m}_t$ for some $t\geq 0$ with the following properties:
		\begin{enumerate}[leftmargin =*]
			\item $\mathrm{Frac}(R_0)=F$,  
			\item $R_i/\mathfrak{m}_i\cong \mathrm{Frac}(R_{i+1})$
			\item $R_t/\mathfrak{m}_t\cong k$.
		\end{enumerate}
		One can take a regular sequence $(a_0,...,a_{t-1})$ of generators for $\mathfrak{m}$ and define $R_i=(R/(a_0,...,a_{i-1}))_{(a_i)}$ (cf.\ \cite[\href{https://stacks.math.columbia.edu/tag/00NQ}{Tag 00NQ}, \href{https://stacks.math.columbia.edu/tag/0AFS}{Tag 0AFS}]{stacks-project}). Now the valuative criterion for properness \cite[Theorem 4.7]{MR213368} shows \[(X_m)_{R_i}(\mathrm{Frac}(R_i))\neq \emptyset \implies (X_m)_{R_{i+1}}(\mathrm{Frac}(R_{i+1}))\neq \emptyset.\] One can conclude by induction.
	\end{proof}
	
	\begin{exmp}
		The assumption that $R$ is regular cannot be dropped from the statement of Lemma \ref{lem: spec}. Here's an example from \cite[\S4]{MR4321613}. Fix a field $k$. Let $X/k$ be any Severi--Brauer variety having $X(k)=\emptyset$. Let $x\in X$ be a closed point. Consider the pushout $\tilde{X}$ in the cocartesian diagram below.
		\[\begin{tikzcd}x\arrow{r}\arrow{d} &\mathrm{Spec}(k)\arrow{d}\\ X\arrow{r} &\tilde{X}\end{tikzcd}\] Let $\tilde{x}\in \tilde{X}$ denote the canonical (singular) $k$-rational point of $\tilde{X}$ and $\mathcal{O}_{\tilde{X},\tilde{x}}$ the local ring. If $A$ is the central simple algebra associated to $X$, then the Azumaya algebra $A\otimes_k \mathcal{O}_{\tilde{X},\tilde{x}}$ is split over the generic point and nontrivial over the closed point by construction.
	\end{exmp}
	
	\bibliographystyle{amsalpha}
	\bibliography{bib}
\end{document}